\documentclass[12pt]{amsart}
\usepackage{amsmath}
\usepackage{amsfonts}
\usepackage{amsthm}
\usepackage{amssymb}
\usepackage{enumerate}
\usepackage{tikz}

\newcommand{\N}{\mathbb{N}}

\newcommand{\al}{\alpha}
\newcommand{\be}{\beta}

\newcommand{\ep}{\epsilon}

\newcommand{\si}{\sigma}

\newcommand{\E}{\emptyset}

\newcommand{\ov}{\overline}
\newcommand{\se}{\subseteq}
\newcommand{\ls}{\langle}
\newcommand{\rs}{\rangle}
\newcommand{\uhr}{\upharpoonright}

\newcommand{\Lf}{\underleftarrow{\lim}\,{\Bold f}}
\newcommand{\Lg}{\underleftarrow{\lim}\,{\Bold g}}
\renewcommand{\L}{\underleftarrow{\lim}}

\def\Int{\mathop\mathrm{Int}}

\def\Per{\mathop\mathrm{Per}}

\def\Bold{\boldsymbol}

\newtheorem{presentment}{}

\theoremstyle{plain}
\newtheorem{theorem}[presentment]{Theorem}
\newtheorem{proposition}[presentment]{Proposition}
\newtheorem{lemma}[presentment]{Lemma}
\newtheorem{corollary}[presentment]{Corollary}

\theoremstyle{definition}
\newtheorem{definition}[presentment]{Definition}
\newtheorem{example}[presentment]{Example}

\theoremstyle{remark}

\title{A characterization of a map whose inverse limit is an arc}
\date{\today}

\author{Sina Greenwood}
\thanks{Sina Greenwood is supported by the Marsden Fund Council from Government funding, administered
	 by the Royal Society of New Zealand.}
\address{University of Auckland, Private Bag 92019, Auckland, New Zealand}
\email{sina@math.auckland.ac.nz}

\author{Sonja \v Stimac}
\thanks{Sonja \v Stimac is supported in part by the Croatian Science Foundation grant IP-2018-01-7491.}
\address{Department of Mathematics, Faculty of Science, University of Zagreb,
	Bijeni\v cka 30, 10\,000 Zagreb, Croatia
}
\email{sonja@math.hr}


\begin{document}

\begin{abstract}
For a continuous function $f:[0,1] \to [0,1]$ we define a splitting sequence admitted by $f$ 
and show that the inverse limit of $f$ is an arc if and only if $f$ does not admit a splitting 
sequence.

\end{abstract}

\maketitle

\baselineskip=18pt


\section{introduction}

In this paper we solve a more than 50 year old open problem about a characterisation of 
a single bonding map on an interval whose inverse limit is an arc. Although at 
first glance the problem seems purely topological, it is also important in  dynamical systems 
since, by \cite{BM1}, every inverse limit space of an interval map can be realised as a global 
attractor for a homeomorphism of the plane. Therefore, our result sheds light on homeomorphisms 
of the plane whose attractors are arcs. In addition, on our way to proving the main result, we 
give dynamical properties, interesting in their own right, of a map on an interval whose 
inverse limit is an arc.

In 1968 Rogers \cite{R} considered the class of single bonding maps on $[0, 1]$ that are nowhere 
strictly monotone and showed that the inverse limit of such a function can be an arc. In the same 
paper Rogers asked a very natural question: \emph{what kind of maps will yield an arc, or more 
specifically, what kind of single bonding map will yield an arc}?

The question turned out to be very hard and has been studied by a number of authors. In 1995 Block 
and Schumann \cite{BS} characterised a unimodal map whose inverse limit is an arc. They showed that 
if $f$ is a unimodal map then its inverse limit is an arc if and only if either $f$ has more than 
one fixed point and no points of other periods, or $f$ has a single fixed point, a period 2 point, 
and no points of other periods. They also gave an example which shows that their characterisation 
for the unimodal maps cannot be extended to piecewise monotone maps. In addition they proved that 
if the inverse limit of a continuous map $f$ on the interval is an arc, then all periodic points of 
$f$ are either fixed points or have period two. 

In 2004 Mo, Shi, Zeng and Mai \cite{MSZM} considered piecewise monotone functions of type N on 
$[0, 1]$ and gave a characterisation of a single type N bonding map whose inverse limit is an arc.

Very recently (2020) Anu\v si\'c and \v Cin\v c \cite{AC} obtained a characterisation of a piecewise 
monotone map whose inverse limit is an arc.

We introduce the very simple notion of a tight sequence (Definition \ref{def-tight}) and study 
a subclass of tight sequences that we call splitting sequences (Definition \ref{def-split}).
We prove that the inverse limit of a continuous surjective function $f$ on an interval is an 
arc if and only if $f$ does not admit a splitting sequence (Theorem \ref{thm-main}). We also 
prove that $f$ admits a splitting sequence if there are two disjoint intervals whose images 
coincide and one of them, $A$, has a subinterval $D \subset A$ such that $f^k(D) = A$ for some 
positive integer $k$ (Lemma \ref{lem-A-B}). This criterion is easy to check for a large class of 
continuous functions (especially if $k$ is small). Additionally, we show that if $f$ has a 
periodic point of period greater than two, then $f$ has a spliting sequence 
(Lemma \ref{lem-2cyclesonly}). This, together with our main theorem, implies the above mentioned 
result from \cite{BS} about a continuous map whose inverse limit is an arc (that  all of its periodic 
points are either fixed points or have period two). 

As shown in \cite{BS}, an inverse limit may not be an arc even if its periodic points have period 
no greater than two. There are maps that have only fixed points, but yield complex inverse limit spaces. 
As we show in this paper, the reason is a splitting sequence. In the Block - Schumann example a 
splitting sequence is easily recognized using the criterion from Lemma \ref{lem-A-B}, as we show in 
Example \ref{ex-B-S}.

The other very interesting example is the Henderson map \cite{H}. It has only two fixed points and 
no points of other periods, but its inverse limit space is the pseudo-arc. The Henderson map is not 
piecewise monotone, so the criterion from \cite{AC} does not work for it. But the existence of a 
splitting sequence for the Henderson map is not hard to prove, as we show in Example \ref{HM}.

On our way towards the main result we also prove that a continuous function $f$ which has at least 
two different periodic orbits of period two, and has an arc as its inverse limit, also has the 
following very interesting property: If $\{s, t\}$ and $\{u, v\}$ are two 2-cycles with $s < t$ and 
$u < v$, then $s < u$ implies $v < t$ (Lemma \ref{2-cycle-order}). Moreover, $f$ has exactly one 
fixed point (Lemma \ref{lem-claim3}). 

The paper is organized as follows: In Section \ref{prelim} we give definitions and define notation 
required in the sequel. In Section \ref{sec-split} we define tight sequences, introduce splitting 
sequences and discuss properties of functions on an interval that do not admit a splitting sequence, 
and which are the base for the proof of our main theorem. In Section \ref{sec-arcs} we prove our 
main theorem. 


\section{Preliminaries}\label{prelim}

A \emph{continuum} is a nonempty compact connected metric space. Let $X$ be a continuum and $
p\in X$ a point. Then $p$ is a \emph{separating point} if $X \setminus \{p\}$ is disconnected. 
A continuum $X$ is an \emph{arc} if $X$ has exactly two nonseparating points called 
\emph{endpoints}. 

For each $n \in \N$, let $X_n$ be a closed interval and $f_{n+1} : X_{n+1} \to X_n$ a continuous 
function. The \emph{inverse limit} of $(f_n)_{n \in \N}$ is the space
$$\L(X_n, f_n) = \{ (x_0, x_1, \ldots) \in \prod_{n \in \N}X_n : \forall\,n\in\N,x_{n+1}\in f(x_n)\}$$
with the topology inherited from the product space $\prod_{n \in \N}X_n$. The functions $f_n$ are 
called \emph{bonding functions}. An inverse limit of continua is a continuum \cite{N}. We are 
concerned with inverse limits of functions $f : [0,1] \to [0,1]$. Denote the inverse limit of a 
single bonding function $f$ by $\Lf$. Bold symbols represent members of $[0,1]^\N$, for example
$\Bold x = (x_0, x_1,\ldots)$. Denote the graph of a function $f$ by $\Gamma(f)$.

Barge and Martin give the following characterization of an endpoint of an inverse limit $\Lf$ for 
a function $f : [0, 1] \to [0, 1]$.

\begin{theorem}\cite[Theorem 1.4]{BM2}\label{marcy}
Let $f : [0,1] \to [0,1]$ be a continuous function. Then $\Bold p$ is an endpoint of $\Lf$ if and 
only if for each integer $n$, each closed interval $J_n = [a_n, b_n]$ with $p_n \in (a_n, b_n)$, 
and each $\ep > 0$, there is a positive integer $k$ such that if $p_{n+k} \in J_{n+k}$ and 
$f^k(J_{n+k}) = J_n$, then $p_{n+k}$ does not separate 
$$(f^k \uhr J_{n+k})^{-1}([a_n, a_n+\ep])\textrm{ and }(f^k \uhr J_{n+k})^{-1}([b_n-\ep,b_n])$$ 
in $[a_{n+k}, b_{n+k}]$ $($$f^k$ is \emph{$\ep$-crooked with respect to $p_{n+k}$}$)$.
\end{theorem}
 
We also require the following result by Block and Schumann.

\begin{proposition}\cite[Proposition 3.1]{BS}\label{BS-prop}
Let $f : [0, 1] \to [0, 1]$ be a continuous function. Then $\Lf$ is a point if and only if $f$ 
admits exactly one fixed point and no periodic points.
\end{proposition}

In order to show that the Henderson  map admits a splitting sequence in Example \ref{HM}, 
we will require the following Lemma.

\begin{lemma}\cite[Lemma 1]{H}\label{lem-H}
	There is a map $f : [0, 1] \to [0, 1]$ such that if $[a, b, c, d]$ is an increasing 
	four-tuple of rational numbers in $(0, 1)$ (that is, $0 < a < b < c < d < 1$), then there 
	exists an integer $m$ such that if $n > m$ and $[u, w]$ is an interval such that 
	$f^n([u, w]) = [a, d]$, then $f^n \uhr [u, w]$ is crooked on $[a, b, c, d]$. 
\end{lemma}

By crooked it is meant that $f^n([u, w])$ contains $[a, d]$ and there is in $[u, w]$ either an 
inverse of $c$ under $f^n$ between two inverses of $b$ or an inverse of $b$ under $f^n$ between 
two inverses of $c$. The Henderson map satisfies the above lemma.

For each $m, n \in \N$, $m < n$, denote the sequence of natural numbers from $m$ to $n$ (inclusive) 
by $[m, n]$ and let 
$$G_{m,n}(f) = \{ (x_m, \ldots, x_n) \in \N^{[m,n]} : \forall \, i \in [m, n-1], f(x_{i+1}) = x_i \}.$$

We define projection functions: $$\pi_n : \Lf \to [0, 1] \textrm{  by } \pi_n(\Bold x) = x_n,$$ 
$$\pi_{n+1, n} : \Lf \to \Gamma(f) \textrm{  by } \pi_{n+1, n}(\Bold x) = (x_{n+1}, x_n),$$ 
and, if $m < n$, define
$$\pi_{[m,n]} : \Lf \to G_{m,n}(f) \textrm{  by } \pi_{[m,n]}(\Bold x) = (x_j)_{j \in [m,n]}.$$ 
If $f$ is surjective, each of these projection functions is onto.

A \emph{basic open subset of $\Lf$} is a set of the form: 
$$U = \bigcap \{ \pi^{-1}_{n_j}(U_j) : j \le k \} \cap \Lf,$$
where $\{ k \} \cup \{ n_j : j \le k\} \subset \N$ and each $U_j$ is an open subinterval of $[0, 1]$.


\section{Splitting sequences}\label{sec-split}

In this section we define tight sequences, and splitting sequences which are a subclass of tight 
sequences. We prove a number of lemmas that give properties of splitting sequences required to 
prove our main theorem.

\begin{definition}\label{def-tight}
Let $f : [0, 1] \to [0, 1]$ be a continuous surjective function and 
$$\si = \ls T_n\subsetneq [0, 1] : n \in \N \rs$$
a sequence of closed intervals. If for each $n \in \N$, $f(T_{n+1}) = T_n$ and there exists 
$m \in \N$ such that for each $n > m$, $T_n$ is nondegenerate, then $\si$ is a \emph{tight} 
sequence. The subcontinuum $\L(T_n, f \uhr T_n)$ is denoted $L(\si)$.
\end{definition}

\begin{definition}\label{def-gen}
Let $f : [0, 1] \to [0, 1]$ be a surjective continuous function. Let $\Bold p \in \Lf$, 
$m \in \N$ and $[a, b] \subset [0, 1]$ be a nondegenerate closed interval such that 
$p_m \in (a,b)$. Let $C \subset \Lf$ be the component of $\pi_m^{-1}([a,b])$ containing 
$\Bold p$. Then $\si = \ls \pi_n(C) : n \in \N \rs$ is a \emph{generated sequence}, or 
more specifically, the \emph{sequence generated by $\Bold p$, $m$ and $[a,b]$}. 
\end{definition}

\begin{lemma}\label{L-exists}
Let $f : [0, 1] \to [0, 1]$ be a continuous surjective function. If $\Bold p \in \Lf$, 
$m \in \N$, $[a, b] \subset [0, 1]$ is nondegenerate, $p_m \in (a, b)$ and $\si$ is the  
sequence generated by $\Bold p$, $m$ and $[a, b]$, then $\si$ is tight.
\end{lemma}

\begin{proof}
Let $C \subset \Lf$ be the component of $\pi_m^{-1}([a,b])$ containing $\Bold p$. First 
observe that $C$ is nondegenerate since $\Bold p \in \Int_{\Lf}(\pi_m^{-1}([a, b]))$ and as 
a component of $\pi_m^{-1}([a, b])$, $C$ must also meet the boundary of $\pi_m^{-1}([a, b])$.

Since $C$ is nondegenerate, for some $m \in \N$, $\pi_m(C)$ is nondegenerate. Thus if 
$n \ge m$ and $\pi_n(C)$ is nondegenerate, then it follows that $\pi_{n+1}(C)$ is nondegenerate 
since $f_{n+1}(\pi_{n+1}(C)) = \pi_n(C)$, and so by induction $\si$ is tight.
\end{proof}

\begin{definition}\label{def-split}
Let $f : [0,1] \to [0,1]$ be a surjective continuous function. If 
$$\si = \ls T_n = [l_n, r_n] : n \in \N \rs$$
is a tight sequence admitted by $f$, $N \se \N$ an infinite set, and 
$$\{S_n \subset [0, 1] : n \in N\}$$
is a collection of nondegenerate closed intervals such that for each $n \in N$,  
$S_n \cap T_n \subset \{l_n,r_n\}$, and $f(S_n) = f(T_n)$, then $\si$ is 
a \emph{splitting sequence admitted by $f$ and witnessed by $\{S_n : n \in N \}$}.
\end{definition}

\begin{example}
If $f : [0,1] \to [0,1]$ is the tent map illustrated in Figure \ref{bucket}, then $f$ admits a splitting 
sequence. Let $T_0 = [\frac14, \frac78]$. If $T_n$ has been defined let $T_{n+1}$ be the component of 
$f^{-1}(T_n)$ contained in $[\frac 12,1]$ and $S_{n+1}$ be the component of $f^{-1}(T_n)$ contained in 
$[0,\frac 12]$. Then $\ls T_n:n\in\N\rs$ is a splitting sequence witnessed by the sets $S_n$.

\begin{figure}[htbp]
\begin{tikzpicture}[scale=0.4]
\draw[gray](0,0) -- (8,0);
\draw[gray](0,0)--(0,8);
\draw[gray](0,8)--(8,8);
\draw[gray](8,0)--(8,8);
\draw[ blue,thick](4,8)--(0,0);
\draw[ blue,thick](4,8)--(8,0);

\draw[gray,dashed](4.5,7)--(0,7);
\draw[gray,dashed](7,2)--(0,2);

\draw[white](0,4.5)
node[black,left]{$T_0$};

\draw[gray,dashed](3.5,7)--(3.5,0);
\draw[gray,dashed](7,2)--(7,0);
\draw[gray,dashed](4.5,7)--(4.5,0);
\draw[gray,dashed](1,2)--(1,0);

\draw[white](2.4,0)
node[black,below]{$S_1$};

\draw[white](5.9,0)
node[black,below]{$T_1$};

\end{tikzpicture}
\caption{Graph of a tent map.}\label{bucket}
\end{figure}
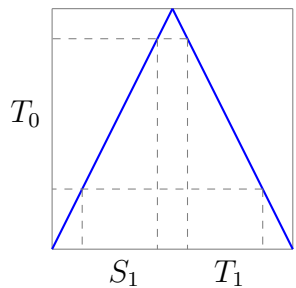
\end{example}

\begin{example}
	The function $f : [0,1] \to [0,1]$ whose graph is shown in Figure \ref{fig-arc} does not admit a 
	splitting sequence. If $\Bold x \in \Lf$ and $x_0 \ne \frac56$, a fixed point of $f$, then $x_n \to 0$. 
	Hence for any tight sequence $\ls T_n = [l_n, r_n] : n \in \N \rs$ there exists $m \in \N$ such that 
	$f^{-1}(r_n) < \frac34$ for every $n > m$ and so there does not exist an interval $S_n \subset [0,1]$ 
	such that $|S_n \cap T_n| \le 1$ and $f(S_n) = f(T_n)$, where $|A|$ denotes the cardinality of a set $A$.
	
	\begin{figure}[htbp]
		\begin{tikzpicture}[scale=0.4]
		\draw[gray](0,0) -- (8,0);
		\draw[gray](0,0)--(0,8);
		\draw[gray](0,8)--(8,8);
		\draw[gray](8,0)--(8,8);
		\draw[ blue,thick](4,8)--(0,0);
		\draw[ blue,thick](4,8)--(8,6);
		
		\draw[gray,dashed](4,8)--(4,0)
		node[below,black]{$\frac12$};
		\draw[gray,dashed](8,6)--(0,6)
		node[left,black]{$\frac34$};
		
		\end{tikzpicture}
		\caption{Graph of a function whose inverse limit is an arc.}\label{fig-arc}
	\end{figure}
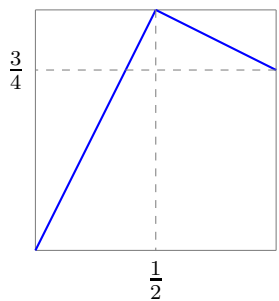
\end{example}

\begin{example}\label{HM}
	Let $f : [0, 1] \to [0, 1]$ be the Henderson map \cite{H}. Recall, $f$ has exactly two fixed points, 0 and 1,
	and for every $x \in (0, 1)$, $f(x) < x$. Its construction is rather complex, but may be described roughly as 
	starting with $g(x) = x^2$ and notching its graph with an infinite set of non-intersecting v-shape notches
	which accumulate at $(1, 1)$. The map $f$ is continuous and $\Lf$ is the pseudo-arc.
	
	We will show that $f$ has a splitting sequence. Let $[a_0, b_0, c_0, d_0]$  be an increasing four-tuple of 
	rational numbers in $(0,1)$. Let $T_0 = [b_0, c_0]$. By Lemma \ref{lem-H}, there exist increasing sequences 
	$\ls n_k \in \N : k \in \N \rs$, $n_k < n_{k+1}$, and $\ls [u_k, w_k] \subset (0, 1) : k \in \N \rs$, 
	$w_k < u_{k+1}$, such that $f^{n_k}([u_k, w_k]) = [a_0, d_0]$ and is crooked on $[a_0, b_0, c_0, d_0]$. Note 
	that $f^{n_k - n_{k-1}}([u_k, w_k]) = [u_{k-1}, w_{k-1}]$. For every $k \in \N$ choose closed intervals 
	$T_{n_k}$ and $S_{n_k}$ in $[u_k, w_k]$ such that 
	$$f^{n_k - n_{k-1}}(T_{n_k}) = f^{n_k - n_{k-1}}(S_{n_k}) = T_{n_{k-1}},$$
	and $|T_{n_k} \cap S_{n_k}| \le 1$, and observe that $f^{n_k}(T_{n_k}) = [b_0, c_0]$. Such choice is possible
	since $f^{n_k} \uhr [u_k, w_k]$ is crooked on $[a_0, b_0, c_0, d_0]$, meaning that there is in $[u_k, w_k]$ 
	either an inverse of $c_0$ under $f^{n_k}$ between two inverses of $b_0$ or an inverse of $b_0$ under $f^{n_k}$ 
	between two inverses of $c_0$. Hence for each $k$, $f^{-n_k}((b_0,c_0))$ has three components, and so 
	$f^{-(n_k-n_{k-1})}(\Int(T_{n_{k-1}}))$ has three components.
	
	For each $k \in \N$ and $j$, $0 < j < n_k - n_{k-1}$, let $T_{n_k - j} = f^j(T_{n_k})$. Then $\ls T_n : n \in \N \rs$ 
	is a splitting sequence witnessed by $\ls S_{n_k} : k \in \N \rs$.
\end{example}

\begin{lemma}\label{lem-2cyclesonly}
Let $f : [0,1] \to [0,1]$ be a surjective continuous function. If $f$ admits a periodic point with 
period $m$ for any $m > 2$ then $f$ admits a splitting sequence.
\end{lemma}

\begin{proof}
In this proof we may write a closed interval $[a, b]$ if we do not know whether $a < b$ or $b < a$ 
and it is assumed to be the appropriate nonempty closed interval.

Suppose $x_0$ is a periodic point with period $m > 2$ and for each $i < m$, $f^i(x_0) = x_i$. Without 
loss of generality suppose that $x_0 = \min\{x_n : n < m\}$. Then $x_1 = f(x_0) > x_0$, and 
$f(x_{m-1}) = x_0 < x_1$.

Suppose $f(x_1) > x_1$. If $x_0 < x_{m-1} < x_1$, since $f(x_{m-1}) < f(x_0) < f(x_1)$ there are 
closed intervals $A \se [x_0, x_{m-1}]$ and $B \subset [x_{m-1}, x_1]$ such that 
$f(A) = f(B) = [f(x_{m-1}), f(x_0)]$, see Figure \ref{diag-2cyclesonly}. Moreover, for each $i < m$ 
there is a closed subinterval $A_i$ of $[x_i, x_{i-1}]$ such that $f(A_i) = [f(x_i), f(x_{i-1})]$.

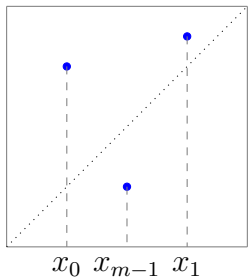
\begin{figure}[htbp]
\begin{tikzpicture}[scale=0.4]
\draw[gray](0,0) -- (8,0);
\draw[gray](0,0)--(0,8);
\draw[gray](0,8)--(8,8);
\draw[gray](8,0)--(8,8);
\filldraw[fill=blue](2,6)[blue]circle(1.2mm);
\filldraw[fill=blue](4,2)[blue]circle(1.2mm);
\filldraw[fill=blue](6,7)[blue]circle(1.2mm);

\draw[gray,dashed](2,6)--(2,0)
node[below,black]{$x_0$};
\draw[gray,dashed](4,2)--(4,0)
node[below,black]{$x_{m-1}$};
\draw[gray,dashed](6,7)--(6,0)
node[below,black]{$x_1$};

\draw[ dotted](0,0)--(8,8);

\end{tikzpicture}
\caption{Graph showing $f(x_0), f(x_1)$ and $f(x_{m-1})$ for the first case of lemma \ref{lem-2cyclesonly}}
\label{diag-2cyclesonly}
\end{figure}

Let $T_0 = [f(x_{m-1}), f(x_0)]$ and $T_1 = A$. If $n \ge 1$ and $T_n$ has been defined such that 
for some $i < m$, $T_n \se [x_i, x_{i+1}]$, let $T_{n+1}$ be a subinterval of $[x_{i-1}, x_{i}]$ 
such that $f(T_{n+1}) = T_n$. Then $\si = \ls T_n  :n \in \N \rs$ is tight. For each $n \in \N$ 
there is a set $S_{mn+1} \subset B$ such that $f(S_{mn+1}) = f(T_{mn+1})$ and 
$S_{mn+1} \cap T_{mn+1} \se \{x_{m-1}\}$. Hence $\tau$ is a splitting sequence.
 
The proof in all cases is analogous. We need only show that in each case there are three points 
$x_i, x_j, x_k$ in the cycle such that $x_i < x_j < x_k$ and $f(x_j)$ is either greater than or 
less than both $f(x_i)$ and $f(x_k)$. If $[f(x_i), f(x_j)] \subset [f(x_j), f(x_k)]$ then take the 
sets $A$ and $B$ used to define the sets $T_n$ and $S_n$, to be subintervals of $[x_i, x_j]$ and 
$[x_j, x_k]$ respectively, such that $f(A) = f(B) = [f(x_{i}), f(x_j)]$, and vice versa.

We show that we can always find three points $x_i, x_j, x_k$ as required. If $f(x_1) > x_1$ and 
$x_1 < x_{m-1}$ then we can take $x_i = x_0$, $x_j = x_1$ and $x_k = x_{m-1}$. If $x_1 > f(x_1)$ 
and $x_{m-1} < x_1$ then we can choose $x_i = x_0$, $x_j = x_{m-1}$ and $x_k = x_1$.

Suppose $x_1 < x_{m-1}$. Then $f(x_{m-2}) = x_{m-1} > x_1$, so if $x_0 < x_{m-2} < x_1$, let  
$x_i = x_0$, $x_j = x_{m-2}$ and $x_k = x_1$. If $x_0 < x_1 < x_{m-1} < x_{m-2}$, let 
$x_i = x_0$, $x_j = x_{m-1}$ and $x_k = x_{m-2}$. Finally, if $x_0 < x_1 < x_{m-2} < x_{m-1}$, 
let $x_i = x_0$, $x_j = x_{m-2}$ and $x_k = x_{m-1}$. 
\end{proof}

In the preceding proof we used a certain technique in our construction of splitting sequences. 
As we will frequently require it, the technique is captured in the following lemma.

\begin{lemma}\label{lem-A-B}
Let $f : [0, 1] \to [0, 1]$ be a surjective continuous function. If there exist $k > 0$, closed 
subintervals $A$ and $B$ of $[0, 1]$, such that $f(A) = f(B)$, $|A \cap B| \le 1$, and there is 
a nondegenerate component of $f^{-k}(A)$ in $A$, then $f$ admits a splitting 
sequence.
\end{lemma}

\begin{proof}
Let $T_1 = A$ (and $T_0 = f(A)$). Since there is a nondegenerate component of $f^{-k}(A)$ in $A$, 
we can choose $T_{k+1}$ to be a subinterval of $A$ such that $f^k(T_{k+1}) = T_1$. For 
$k \ge i \ge 2$ let $T_i = f(T_{i+1})$. Obviously $T_1 = f(T_2) = f^k(T_{k+1})$. Analogously, if 
$n > k$, $n = 0\mod k$ and $T_{n-k+1}$ has been defined, let $T_{n+1}$ be a subinterval of $A$ 
such that $f^k(T_{n+1}) = T_{n-k+1}$. For $n \ge i \ge n-k+1$ let $T_i = f(T_{i+1})$. Then 
$\si = \ls T_i : i \in \N \rs$ is a tight sequence. Since for every $n > 0$ $T_{nk+1} \se A$ and 
$f(A) = f(B) = T_0$, for every $n > 0$ we can choose an interval $S_{nk+1} \se B$ such that 
$f(S_{nk+1}) = f(T_{nk+1})$. Thus $\si$ is a splitting sequence. 
\end{proof}

If $A$ and $B$ are intervals and  $k \in \N$ as in Lemma \ref{lem-A-B}, we say that 
\emph{the pair $(A, B)$ generates a splitting sequence of order $k$}.

\begin{example}\label{ex-B-S}
We give an example from \cite{BS} which shows that there exists a piecewise monotone map which has more than one 
fixed point and no points of other periods, but its inverse limit is not an arc.
	
Let $f, g : [0,1] \to [0,1]$ be maps whose graphs are shown in Figure \ref{fig-B-S}. Obviously, the map $g = f^2$ 
has more than one fixed point and no points of other periods. Also, it is well known that $\Lf = \Lg$ and  is
homeomorphic to a $\sin \frac{1}{x}$-continuum \cite{N}. 
	
It is easy to see that the both maps have splitting sequences. We will use the above criterion. Let 
$A = [\frac12, 1]$ and $B = [\frac14, \frac12]$. Then $A \cap B = \{ \frac12 \}$, $f(A) = g(A) = A$ and 
$f(B) = g(B) = A$. Therefore, $(A, B)$ generates a splitting sequence of order $1$.
	
	\begin{figure}[htbp]
		\begin{tikzpicture}[scale=0.4]
		\draw[gray](0,0) -- (8,0);
		\draw[gray](0,0)--(0,8);
		\draw[gray](0,8)--(8,8);
		\draw[gray](8,0)--(8,8);
		\draw[ blue,thick](4,8)--(0,0);
		\draw[ blue,thick](4,8)--(8,4);
		
		\draw[gray,dashed](4,8)--(4,0)
		node[below,black]{$\frac12$};
		\draw[gray,dashed](8,4)--(0,4)
		node[left,black]{$\frac12$};

		\draw[gray](10,0) -- (18,0);
		\draw[gray](10,0)--(10,8);
		\draw[gray](10,8)--(18,8);
		\draw[gray](18,0)--(18,8);
		\draw[ blue,thick](12,8)--(10,0);
		\draw[ blue,thick](12,8)--(14,4);
		\draw[ blue,thick](18,8)--(14,4);
		\draw[gray,dashed](2,8)--(2,0)
		node[below,black]{$\frac14$};
		
		\draw[gray,dashed](14,8)--(14,0)
		node[below,black]{$\frac12$};
		\draw[gray,dashed](18,4)--(10,4)
		node[left,black]{$\frac12$};
		\draw[gray,dashed](12,8)--(12,0)
		node[below,black]{$\frac14$};
		
		\end{tikzpicture}
		\caption{Graphs of functions $f$ (left) and $g = f^2$ (right) whose inverse limits are the
			$\sin \frac{1}{x}$-continuum.}\label{fig-B-S}
	\end{figure}
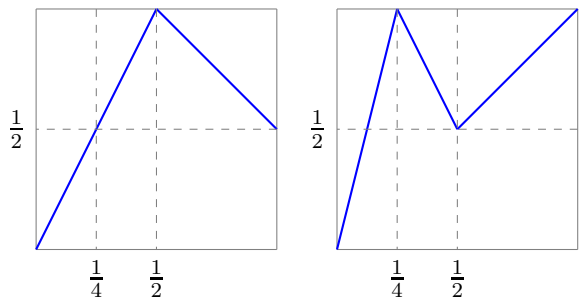
\end{example}

\begin{lemma}\label{lem-no-extreme-pt}
Let $f : [0,1] \to [0,1]$ be a continuous function such that $f$ does not admit a splitting 
sequence. If $0 \le d < e \le 1$ and either $d$ and $e$ are fixed points or $\{d, e\}$ is a 
2-cycle, then there is exactly one component $C$ of $f^{-1}((d, e))$ such that $f(C) = (d, e)$.
\end{lemma}

\begin{proof}
Since either $d$ and $e$ are fixed points or $\{d, e\}$ is a 2-cycle, there exists a component 
$C \subseteq [d, e]$ of $f^{-1}((d, e))$ such that $f(C) = (d, e)$. Suppose that (for either 
case), $f^{-1}((d, e))$ has a second component $D$ such that $f(D) = (d, e)$. Then 
$f(\ov C) = f(\ov D) = [d, e]$, $|\ov C\cap\ov D| \le 1$ and there is a nondegenerate component 
of $f^{-1}(C)$ in $C$. Thus the pair $(\ov C, \ov D)$ generates a splitting sequence of order 1.
\end{proof}

\begin{corollary}\label{cor-fixed-points}
Let $f : [0, 1] \to [0, 1]$ be a surjective continuous function such that $f$ does not admit 
a splitting sequence. If $F$ is the set of fixed points admitted by $f$ and $d$ is an accumulation
point of $F$, then 
$$\Lf=\L([0,d],f\uhr[0,d])\cup\L([d,1],f\uhr[d,1]).$$
and $$\L([0,d],f\uhr[0,d])\cap\L([d,1],f\uhr[d,1])=\{(d,d,\ldots)\}.$$
\end{corollary}

\begin{lemma}\label{2-cycle-order}
Let $f : [0, 1] \to [0, 1]$ be a surjective continuous function such that $f$ does not admit a 
splitting sequence. If $f$ admits two 2-cycles $\{s, t\}$ and $\{u, v\}$ with $s < t$ and $u < v$, 
then either $s < u < v < t$ or $u < s < t < v$.
\end{lemma}

\begin{proof}
Suppose $s < u < t < v$. Then there are closed intervals $A \se [s, u]$ and $B \subset [u, t]$ 
such that $f(A) = f(B) = [t, v]$. Also, there is an interval $A' \subset [t, v]$ such that 
$f(A') = [s, u]$. Thus $(A, B)$ generates a splitting sequence of order 2. Similarly if 
$s < t < u < v$, $u < s < v < t$, or $u < v < s < t$.
\end{proof}

\begin{lemma}\label{lem-one-2-cycle}
Let $f : [0, 1] \to [0, 1]$ be a surjective continuous function such that $f$ does not admit a 
splitting sequence. If $f$ admits two 2-cycles $\{s, t\}$ and $\{u, v\}$ with $s < u$, then 
there is exactly one component $C$ of $f^{-1}([s, u])$ such that $f(C) = [s, u]$, and
 there is exactly one component $C'$ of $f^{-1}([v, t])$ such that $f(C') = [v, t]$.
\end{lemma}

\begin{proof}
Let $C \subseteq [v, t]$ be a component of $f^{-1}([s, u])$ such that $f(C) = [s, u]$. Since
$f([s, u]) \supseteq [v, t]$, we can choose a nondegenerte component of $f^{-2}(C)$ in $C$.
If $f^{-1}([s, u])$ has a second component $D$ such that $f(D) = [s, u]$ and $|C \cap D| \le 1$, 
the pair $(C, D)$ generates a splitting sequence of order 2. The proof of the second statement is 
analogous.
\end{proof}

\begin{lemma}\label{lem-f-2}
Let $f : [0, 1] \to [0, 1]$ be a surjective continuous function. Then $f$ admits a splitting 
sequence if and only if $f^2$ admits a splitting sequence. 
\end{lemma}

\begin{proof}
Let $N \subset \N$ be an infinite set, $\ls T_n : n \in \N \rs$ a splitting sequence admitted by 
$f$ and  witnessed by $\{ S_n : n \in N \}$. Let  $\si = \ls T_{2n} : n \in \N \rs$ and let 
$\tau = \ls T_{2n+1} : n \in \N \rs$. Observe that either the set of even values in $N$ is infinite, 
or the set of odd values is. If the even values are infinite then $\si$ is a splitting sequence 
admitted by $f^2$ and witnessed by $\{ S_n : n \in N, \ n \textrm{ is even} \}$. If the set of 
odd values of $N$ is infinite then the $\tau$ is a splitting sequence admitted by $f^2$ and witnessed 
by $\{ S_n : n \in N, \ n \textrm{ is odd} \}$.
 
Suppose $\ls R_n : n \in \N \rs$ is a splitting sequence admitted by $f^2$ and  witnessed by 
$\{ S_n : n \in N \}$ for some infinite set $N$. For each $n$ let $T_{2n} = R_n$ and 
$T_{2n+1} = f(R_{n+1})$. For each $n \in N$ let $S'_{2n} = S_n$. Then $\ls T_n : n \in \N \rs$ is a 
splitting sequence admitted by $f$ and witnessed by $\{ S'_{2n} : n \in N \}$. 
\end{proof}

For the remainder of this paper, given a function $f : [0,1] \to [0,1]$, let $a = \max(f^{-1}(0))$ and $b = \min(f^{-1}(1))$.
 
\begin{lemma}\label{lem-claim1}
Let $f : [0,1] \to [0,1]$ be a surjective continuous function that does not admit a splitting sequence.
Let $d$ be the maximum fixed point of $f$. Suppose $a < b$. Then the following hold:
\begin{enumerate}[(i)]
\item $d$ is the only fixed point  in $[b, 1]$;
	
\item $f([b, 1]) \subset (b, 1]$;
	
\item $f \uhr [b, 1]$ does not admit a 2-cycle; and
	
\item $\L([b, 1], f\uhr[b,1]) = \{( d, d, \ldots )\}$.
\end{enumerate}
\end{lemma}

\begin{proof} Observe that if $f(1) = 1$ then $d=1$.
\begin{enumerate}[(i)]
\item Suppose $d' \in [b, 1]$, $d'$ is a fixed point and $d' < d$.  Then there are intervals $A \se [d', d]$ and $B \subset [0, b]$ such that 
	$f(A) = f(B) = [d', d]$, contradicting Lemma \ref{lem-no-extreme-pt}.

\item If $b \in f([b, 1])$ then there exist $A \subseteq [b, 1]$ and $B \subseteq [a, b]$ such that
	$f(A) = f(B) = [b, 1]$. Since $f^{-1}([b, 1]) \supseteq [b, 1]$, by Lemma \ref{lem-A-B} $f$ admits
	a splitting sequence, a contradiction.
	
\item The statement follows from Lemma \ref{lem-no-extreme-pt} since if  $\{ p, q \}$ is a 2-cycle 
	admitted by $f \uhr [b, 1]$, $p < q$, then there are  intervals $A \se [p, q]$ and 
	$B \subset [0, b]$ such that $f(A) = f(B) = [p, q]$. 

\item By (i), (iii), Proposition \ref{BS-prop} and Lemma \ref{lem-2cyclesonly}, 
	$\L([b, 1], f \uhr [b, 1])$ is a singleton, and as $d$ is a fixed point, 
	$ \L([b, 1], f \uhr [b, 1])=\{(d, d, \ldots)\}$. 
\end{enumerate}
\end{proof}

Analogously to Lemma \ref{lem-claim1} we can show the following:

\begin{lemma}\label{lem-claim2}
	Let $f : [0,1] \to [0,1]$ be a surjective continuous function that does not admit a splitting 
	sequence. Let $e$ be the minimum fixed point of $f$. Suppose $a < b$. Then the following hold:
	\begin{enumerate}[(i)]
	\item $e$ is the only fixed point in $[0,a]$;
	
	\item $f([0,a])\subset[0,a)$;
	
	\item $f\uhr[0,a]$ does not admit a 2-cycle; and
	
	\item $\L([0, a], f \uhr [0, a]) = \{ (e, e, \ldots) \}$.
	\end{enumerate}
\end{lemma}

\begin{lemma}\label{lem-claim3}
Let $f : [0,1] \to [0,1]$ be a surjective continuous function that does not admit a splitting sequence.
Suppose $b < a$. Let $d$ be a fixed point between $b$ and $a$. Let $a' = \min(f^{-1}(0))$, 
$b' = \max(f^{-1}(1))$, and let 
$$r = \max \{ x \in (a', b') : f(x) \in (a', b') \textrm{ and } x \textrm{\ is periodic with } \Per(x) \le 2\}.$$ 
Then the following hold:
\begin{enumerate}[(i)]
\item for every $x \in [0, b']$, $f(x) > r$ and for every $x \in [a', 1]$, $f(x) < f(r)$;
		
\item the function $f$ admits exactly one fixed point; and
		
\item $f$ admits a unique 2-cycle $\{s, t\}$ such that  $s < t$, and either $0 \le s < b'$ or $a' < t \le 1$.
\end{enumerate}
\end{lemma}

\begin{proof}
\begin{enumerate}[(i)]
\item 
If there exists $x \in [0, b']$ such that $f(x) \le r$, then there are closed intervals $A \se [b',f(r)]$ 
and $B \subset [x,b']$ such that $f(A) = f(B) = [r, 1]$. There is an interval $A' \subset [r, 1]$ such that 
$f(A') = [b', f(r)]$. Thus $(A, B)$ generates a splitting sequence of order 2.

Analogously, if there exists $x \in [a', 1]$ such that $f(x)\ge f(r)$, then we can obtain a splitting sequence 
of order 2.

\item
Suppose $f$ admits a second fixed point $e$ and $d < e$. Then by (i), $b' < d < e < a'$, and $f^{-1}([d, e])$ 
has a component $C \subset [b', d]$ and a component $D \se [d, e]$ such that $f(C) = f(D) = [d, e]$, 
contradicting Lemma \ref{lem-no-extreme-pt}.

\item
Since $f([b', a']) = [0, 1]$, the claim follows from Lemmas \ref{lem-no-extreme-pt}, \ref{2-cycle-order} 
and \ref{lem-one-2-cycle}.
\end{enumerate}
\end{proof}

\begin{proposition}\label{lem-endpoints}
If $f : [0,1] \to [0,1]$ is a surjective continuous function that does not admit a splitting 
sequence, then  either
\begin{enumerate}
\item[(a)]
$f$ admits at least 2 fixed points and if $d$ is the maximum and $e$ the minimum fixed point, 
then $(d, d, d, \ldots)$ and $(e, e, e,\ldots)$ are endpoints of $\Lf$; or
\item[(b)]
$f$ admits a 2-cycle and if $\{s, t\}$ is a 2-cycle such that for any other 2-cycle $\{u, v\}$, 
$s < u$, 
then $(s, t, s, t, \ldots)$ and $(t, s, t, s,\ldots)$ are endpoints of $\Lf$.
\end{enumerate}
\end{proposition}

\begin{proof}
We consider two cases:
\begin{enumerate}
\item $a < b$ and
\item $b < a$. 
\end{enumerate}
Case (1): We  first show that $(d, d, \ldots)$ is an endpoint of $\Lf$ which we do by applying 
Lemma \ref{lem-claim1} and Theorem \ref{marcy}.

Let $\ep > 0$ and let $J_{0} = [\al_0, \be_0]$ be an interval such that $d \in (\al_0, \be_0)$. 
By Lemma \ref{lem-claim1} (iv), for every  $\Bold x \in \Lf$ with $x_0 \in [b, 1] \setminus \{ d \}$, 
there exists $j$ such that $x_j < b$, and so by Lemma \ref{lem-claim1} (ii), $x_n < b$ for every 
$n > j$. Hence there exists $k$ such that $f^{-k}(\al_0) \subset [0, b)$ and $f^{-k}(\be_0) \subset [0, b)$, 
and therefore $f^{-k}([\al_0, \al_0 + \ep]) \cap [0, b] \ne \E$ and 
$f^{-k}([\be_0 - \ep, \be_0]) \cap [0, b] \ne \E$. Thus, as $[0, b]\subset [0, d)$, $f^k$ is $\ep$-crooked 
with respect to $(d, d, \ldots)$. 

By applying Lemma \ref{lem-claim2} and Theorem \ref{marcy} we can analogously establish that 
$(e ,e, \ldots)$ is an endpoint of $\Lf$.

Case (2): Let $d$ be a fixed point between $b$ and $a$. We now show that $(s, t, s, t, \ldots )$ and 
$(t, s, t, s, \ldots )$ 
are endpoints. 

Let $g = f^2$. By Lemma \ref{lem-f-2}, $g$ does not admit a splitting sequence. Since $f$ admits a 
2-cycle, $g$ admits at least two fixed points, and hence by Lemma \ref{lem-claim3} (ii), $g$ must 
satisfy the condition of case (1). Thus $g$ admits at least three fixed points $d$, $s'$ and $t'$, 
such that $d$ is the fixed point guaranteed by Lemma \ref{lem-claim3} (ii), $s'$ is the minimum and 
$t'$ the maximum fixed point admitted by $g$. Hence $s' < d < t'$. It follows from Lemma \ref{lem-claim3} 
(ii) and (iii) and Lemma \ref{2-cycle-order}, that $\langle s', t' \rangle = \langle s, t \rangle$.

Now the function $h : \Lf \to \Lg$ defined by $$h((x_0, x_1, x_3, \ldots )) = (x_0, x_2, x_4, \ldots )$$
is a homeomorphism, so 
$$(s, t, s, t, \ldots ) = h^{-1}((s ,s, \ldots )) \textrm{ and } (t, s, t, s, \ldots ) = h^{-1}((t ,t, \ldots ))$$  
are endpoints of $\Lf$.

Thus, if case (1) holds we have two fixed points that determine two endpoints of $\Lf$, and if case (2) 
holds we have a 2-cycle that determines two endpoints as required.
\end{proof}

\begin{lemma}\label{one-boundary-cycle}
Let $f : [0, 1] \to [0, 1]$ be a continuous surjective function that does not admit a splitting sequence. 
If $f$ admits a 2-cycle $\{s,t\}$  such that $s\in\{0,1\}$, then $(s,t,s,t,\ldots)$ and $(t,s,t,s,\ldots)$ 
are endpoints.
\end{lemma}

\begin{proof}
Suppose $s = 1$. The proof is similar if $s = 0$. Observe that $t \ge b = \min(f^{-1}(1))$, and hence $f$ 
does not satisfy Lemma \ref{lem-claim1} (iii) which states that the function $f \uhr[b, 1]$ does not admit 
a 2-cycle. Hence $f$ satisfies the condition $b < a$ (case (2) in the proof of Proposition \ref{lem-endpoints}). 
Since $s = 1$, by Lemma \ref{2-cycle-order}, $\{ s, t \}$ is the 2-cycle determining the two endpoints of 
Proposition \ref{lem-endpoints} (b).
\end{proof}

\begin{lemma}\label{lem:p-not-in-E}
Let $f : [0, 1] \to [0, 1]$ be a continuous surjective function such that $f$ does not admit a splitting 
sequence, and $f$ admits at least two fixed points. If $\si = \ls T_n = [l_n, r_n] : n \in \N \rs$ is 
the  sequence generated by a point $\Bold p$, $m \in \N$ and interval $[c,d]$, and $[c,d]$ does not contain 
the maximum or minimum fixed point, then there exists $k \in \N$ such that for every $n \ge k$, 
$p_n \not\in \{ l_n, r_n \}$.
\end{lemma}

\begin{proof}
By Lemma \ref{lem-claim3} (ii), $f$ satisfies the requirement of case (1) in the proof of Proposition 
\ref{lem-endpoints}. Thus $a = \max f^{-1}(0) < b = \min(f^{-1}(1))$ and by Lemma \ref{lem-claim1} (iv) 
and Lemma \ref{lem-claim2} (iv), there exists $j \in \N$ such that for each $n > j$, $T_n \subset [a,b]$.

By Lemma \ref{L-exists}, $\si$ is tight so there exists $r > j$ such that for every $n > r$, $T_n$ is 
nondegenerate. Let $$N = \{ n > r : p_n \in \{ l_n, r_n \} \},$$ and suppose that $N$ is infinite. For every 
$n \in \N$, let $[l'_{n+1}, r'_{n+1}]$ be the component of $f^{-1}(T_n)$ containing $p_{n+1}$. Since 
$p_m \in (c, d) = \Int T_m$, for every $n > m$ we have that $p_n \not\in \{ l'_n, r'_n \}$. If $n \in N$ 
then either 
$$f([l'_{n+1}, r'_{n+1}]) = [l'_{n}, p_{n}] = T_n \textrm{ or } f([l'_{n+1}, r'_{n+1}]) = [p_{n}, r'_{n}] = T_n.$$
 
Then for $n \in N$ we have that $T_n \subset [a, b]$ and $p_n \in \{ l_n, r_n \}$, and hence we can choose two 
sets $A_{n+1} \subset [l'_{n+1}, p_{{n+1}}]$ and $B_{n+1} \subset [p_{n+1}, r'_{n+1}]$ such that 
$f(A_{n+1}) = f(B_{n+1}) = T_{n}$ and $A_{n+1} \cap B_{n+1} \se \{ p_{n+1} \}$. 

Let $R_0 = T_0$. If $n \ge 0$ and $R_{n}$ has been defined, let $R_{n+1}$ be a subinterval of either 
$T_{n+1} \cap A_{n+1}$ or $T_{n+1} \cap B_{n+1}$ if $n \in N$, otherwise let $R_{n+1}$ be any subinterval 
of $T_{n+1}$, and in each case such that $f(R_{n+1}) = R_{n}$. For each $n \in N$, if $T_{n+1} \subset A_{n+1}$ 
let $S_{n+1}$ be a subinterval of $B_{n+1}$, and if $T_{n+1} \subset B_{n+1}$ let $S_{n+1}$ be a subinterval 
of $A_{n+1}$, such that $f(S_{n+1}) = R_n$. Thus $\ls R_n : n \in \N \rs$ is a splitting sequence, a contradiction.
\end{proof}


\section{Arcs}\label{sec-arcs}

\begin{lemma}\label{lem-2endpoints}
If $f : [0, 1] \to [0, 1]$ is a continuous surjective function with exactly two fixed points, and $f$ 
does not admit a splitting sequence, then $\Lf$ is an arc.
\end{lemma}
 
\begin{proof}
Suppose $d$ and $e$ are the only fixed points, $e < d$. Since $f$ is surjective and does not admit a 
splitting sequence, either $e = 0$ or $d = 1$. Suppose that $e = 0$ and $d \ne 1$ (the proof is analogous 
if $e \ne 0$ and $d = 1$, or if $e = 0$ and $d = 1$). Since $f$ is surjective, for every $x \in (0, d)$, 
$f(x) > x$. The conditions of case (1) in the proof of Proposition \ref{lem-endpoints} are satisfied so 
$(0, 0, \ldots)$ and $(d, d, \ldots)$ are endpoints. 

Let $$\Bold p \in \Lf \setminus \{ (0 ,0, \ldots), (d, d, \ldots) \}.$$ We show that $\Bold p$ is a 
separating point. Recall $b = \min \{ x \in [0, 1] : f(x) = 1 \}$. By Lemma \ref{lem-claim1} (iv) it 
follows that for some $m \in \N$, $p_n < b$ for every $n > m$. Let $$N = \{ n > 0 : |f^{-1}(p_{n-1})| > 1 \}.$$

(a) Suppose $N$ is finite. Choose some $m \ge \max(N)$ such that $p_m < b$. Then for every $n > m$, $p_n < b$ 
and $f^{-1}(p_n) = \{ p_{n+1} \}$, so 
$$f^{-1}([0, p_n]) = [0, p_{n+1}] \textrm{ and } f^{-1}([p_n, 1]) = [p_{n+1}, 1].$$ For each $n \in \N$ let 
$X_n = [0, p_{m+n}]$, $Y_n = [p_{m+n}, 1]$, $g_n = f \uhr X_n$ and $h_n = f \uhr Y_n$, and let $X = \L(X_n, g_n)$ 
and $Y = \L(Y_n, h_n)$. Then clearly $$\pi_{[m, \infty)}(\Lf) = X \cup Y \textrm{ and } X \cap Y = \{\Bold p\}.$$
Let $X' = \pi_{[m, \infty)}^{-1}(X)$ and $Y' = \pi_{[m, \infty)}^{-1}(Y)$. Since $\pi_{[m, \infty)}^{-1}$ is the
bijection defined by $$(x_m, x_{m+1}, \ldots) \mapsto (f^{m-1}(x_m), \ldots, f(x_m), x_m, x_{m+1}, \ldots),$$
it follows that $\Lf = X' \cup Y'$ and $X' \cap Y' = \{ \Bold p \}$. Thus $\Bold p$ is a separating point of $\Lf$.

(b) Suppose $N$ is infinite. For all $\ep > 0$ and $i \in \N$ let 
$$\si_{\ep, i} = \ls T^{\ep, i}_n = [a^{\ep, i}_n, b^{\ep, i}_n] : n \in \N \rs$$ be the tight sequence generated 
by $\Bold p$, $i$ and $[p_i - \ep, p_i + \ep]$. Suppose that for some $i$ and $\ep$, there is an infinite set 
$M \se N$ such that for every $n \in M$,  $f^{-1}(T^{\ep, i}_n) \setminus \Int \, (T^{\ep, i}_{n+1})$ has a 
component $C_{n+1}$ with $p_{n} \in f(C_{n+1})$. By Lemma \ref{lem:p-not-in-E} we can assume that for each 
$n \in M$, $p_n \not\in \{ a^{\ep, i}_n, b^{\ep, i}_n \}$. Let $k \in M$ and let $L_k = [a^{\ep, i}_k, p_k]$ 
and $R_k = [p_k, b^{\ep, i}_k]$. If $j \ge k$ and $L_j, R_j$ have been defined, let $L_{j+1}$ and $R_{j+1}$ be 
components of  $T^{\ep,i}_{j+1}$ such that $f(L_{j+1})=L_j$ and $f(R_{j+1})=R_j$. Clearly each of the sets 
$L_{k+1}$ and $R_{k+1}$ contains a different endpoint of $T^{\ep, i}_{k+1}$. If $j\le k$ and $L_j, R_j$ have been 
defined, let $L_{j-1} = f(L_j)$ and $R_{j-1} = f(R_{j})$.
 
Then $\tau_1 = \ls L_n : n \in \N \rs$ and $\tau_2 = \ls R_n : n \in \N \rs$ are tight sequences. Observe that 
for each $n \in M$  there is a subinterval $D_{n+1}$ of $C_{n+1}$ such that either $f(D_{n+1}) = L_n$, or 
$f(D_{n+1}) = R_n$. Then one of the sequences $\tau_1$ or $\tau_2$ is a splitting sequence.
 
Thus we have that for every $\ep$ and $i$, $f^{-1}(p_n) \subset T^{\ep, i}_{n+1}$ for all but finite $n \in N$. 
For every $\ep > 0$ and $i \in \N$ such that $$0, d, 1 \not\in [p_i - \ep, p_i + \ep],$$ choose $m_{\ep, i}$ 
such that $f^{-1}(p_n) \subset T^{\ep, i}_{n+1}$ for all $n \ge m_{\ep, i}$. Thus 
$[0,1] \setminus T^{\ep, i}_{m_{\ep, i}}$ has two components, $A'_{\ep, i}$ and $B'_{\ep, i}$. Let  
$$A_{\ep, i} = \pi^{-1}_{m_{\ep, i}}(A'_{\ep, i}), \ B_{\ep, i} = \pi^{-1}_{m_{\ep, i}}(B'_{\ep, i}),$$
$$A = \bigcup \{ A_{\ep, i} : \ep > 0, i \in \N \textrm{ and } 0, d, 1 \not\in [p_i - \ep, p_i + \ep] \},$$
and $$B = \bigcup \{ B_{\ep, i} : \ep > 0, i \in \N \textrm{ and } 0, d, 1 \not\in [p_i - \ep, p_i + \ep] \}.$$

Then $\Bold p \not\in A \cup B$, $A \cap B = \E$ and, since 
$\bigcap \{ L(\si_{\ep, i}) : \ep > 0, i \in \N \} = \{ \Bold p \}$, $A \cup B \cup \{ \Bold p \} = \Lf$.
Thus $\Bold p$ is a separating point and so $\Lf$ is an arc.
\end{proof}

The next three  lemmas reference the behavior of a function on either side of a fixed point. We define four 
types of fixed point in the following definition in order to simplify the discussions.

\begin{definition}
Suppose that $f : [0, 1] \to [0, 1]$ is a  continuous surjective function and $c, d, e \in [0, 1]$, $c < d < e$. 
If $d$ is a fixed point of $f$, $d$ is the only fixed point in the interval $(c, e)$, either $c = 0$ or $c$ is 
a fixed point, and either $e = 1$ or $e$ is a fixed point, then $d$ is 
\begin{itemize}
\item
an \emph{S-type} fixed point if for each $x \in [c, d]$, $f(x) \le x$, and for each $x \in [d, e]$, $f(x) \ge x$,
\item
an \emph{N-type} fixed point if for each $x \in [c, d]$, $f(x) \ge x$, and for each $x \in [d, e]$, $f(x) \le x$,
\item
an \emph{M-type} fixed point if for each $x \in [c, e]$, $f(x) \ge x$, and 
\item
a \emph{W-type} fixed point if for each $x \in [c, e]$, $f(x) \le x$.
\end{itemize}
In each case the type is \emph{witnessed by $(c,e)$}.
\end{definition}

\begin{lemma}\label{S-fixed}
Suppose that $f : [0, 1] \to [0, 1]$ is a continuous surjective function that does not admit a splitting 
sequence. If $f$ admits a fixed point $d$ that is S-type, M-type or W-type, then 
$$\Lf = \L([0, d],f \uhr [0, d]) \cup \L([d, 1], f \uhr [d, 1]),$$ and 
$$\L([0, d], f \uhr [0, d]) \cap \L([d, 1], f \uhr [d, 1]) = \{ (d, d, \ldots ) \}.$$
\end{lemma}

\begin{proof}
Suppose $d$ is an S-type fixed point witnessed by $(c, e)$. Then by the definition of S-type, $c$ and $e$ are 
fixed points. By Lemma \ref{lem-no-extreme-pt}, $f^{-1}(d) = \{ d \}$ and hence the result follows.

Suppose that $d$ is an M-type fixed point witnessed by $(c, e)$. The proof for a W-type fixed point is analogous. 
Observe that, by the surjectivity of $f$ and Lemma \ref{lem-no-extreme-pt}, $c$ and $e$ are fixed points.

Let $p' = \max(f([0, d])$ and let $p = \max\{ x \in [0, d] : f(x) = p' \}$. By Lemma \ref{lem-no-extreme-pt}, 
$p' < e$. Let $q = \min\{ x \in [d, 1] : f(x) = p' \}$. If $p' = d$ then the result follows from 
Lemma \ref{lem-no-extreme-pt}. Suppose that $p' > d$. Let $A \se [p, d]$ be an interval such that $f(A) = [d, p']$. 
Then $([d, q], A)$ generates a splitting sequence of order 1, and hence $p' = d$. The result follows.
\end{proof}

\begin{lemma}\label{N-fixed}
Suppose that $f : [0, 1] \to [0, 1]$ is a continuous surjective function that does not admit a splitting sequence. 
If $f$ admits an N-type fixed point $d$ witnessed by $(c, e)$, then $\L([c, e], f \uhr [c, e])$ is an arc, and if  
$c$ and $e$ are fixed points, then $(d, d, \ldots )$ is a separating point of $\Lf$.
\end{lemma}

\begin{proof}
Suppose $c$ and $e$ are fixed points. Let $p = \max(f([c, d])$ and $q = \min(f([d, e])$. By 
Lemma \ref{lem-no-extreme-pt}, $c < q$ and $p < e$. Then the functions $f \uhr [c, p]$ and $f \uhr [q, e]$ 
satisfy the conditions of Proposition \ref{lem-endpoints} case (1). Each function  has exactly two fixed 
points and so by Lemma \ref{lem-2endpoints}, each of the sets $A_1 := \L([c, d], f \uhr [c, d])$ and 
$A_2 := \L([d, e], f \uhr [d, e])$ is an arc, and by Lemma \ref{lem-claim1} (iv), 
$A_1 \cap A_2 = \{( d, d, \ldots )\}$.

Suppose $\Bold x \in \L([c, e], f \uhr [c, e]) \setminus \{ (d, d, \ldots ) \}$. If $x_0 \in [c, q)$, then for 
each $n \in \N$, $x_n \in [c, q)$. Hence $\Bold x \in A_1$, and similarly if $x_0 \in (p, e]$ then 
$\Bold x \in A_2$. Suppose $x_0 \in [q, p]$. Since $\Bold x \ne (d, d, \ldots )$ there exists $n \in \N$ such 
that $x_n \not\in [q, p]$. Let $m = \min \{ n \in \N : x_n \not\in [q, p] \}$. If $x_m \in [c, q)$, then   
$x_n \in [c, q)$ for each $n > m$, and hence $\Bold x \in A_1$. Otherwise $\Bold x \in A_2$.

Thus $\L([c, e], f \uhr [c, e]) = A_1 \cup A_2$ and $(d, d, \ldots )$ is a separating point of  
$\L([c, e], f \uhr [c, e])$ and hence of $\Lf$.

If $e$ is not a fixed point, then $e = 1$, and by the surjectivity of $f$ and Lemma \ref{lem-no-extreme-pt}, 
$f \uhr [c,1]$ satisfies the condition of Proposition \ref{lem-endpoints} case (1). Since $d$ is an N-type fixed 
point, if $c \ne 0$, $c$ is either an S-type or an M-type fixed point, or an accumulation point of the set of 
fixed points. Hence by Corollary \ref{cor-fixed-points} and Lemma \ref{S-fixed},
$$\Lf = \L([0, c], f \uhr [0, c]) \cup \L([c, 1], f \uhr [c, 1]),$$ 
$$\L([0, c], f \uhr [0, c]) \cap \L([c, 1], f \uhr [c, 1]) = \{ (c, c, \ldots ) \},$$
and $\L([c, 1], f \uhr [c, 1])$ is an arc since $f \uhr [c, 1]$ admits exactly two fixed points.

Similarly if $c$ is not a fixed point.
\end{proof}

\begin{lemma}\label{lem-arc}
If $f : [0, 1] \to [0, 1]$ is a continuous surjective function that does not admit a splitting sequence, then 
$\Lf$ is an arc.
\end{lemma}

\begin{proof}
Since $\Lf$ is an arc if and only if $\Lf^2$ is an arc, and if $f$ satisfies the condition of case (2) of the 
proof of Proposition \ref{lem-endpoints}, then $f^2$ satisfies the condition of case (1), we can assume that $f$ 
satisfies the condition of case (1) and hence admits more than one fixed point. 

Let $E$ be the set containing the 2 endpoints admitted by $f$ as in Proposition \ref{lem-endpoints}, and let 
$d$ and $e$, $d < e$, be the two fixed points that determine the members of $E$. It remains to show that if 
$\Bold x \in \Lf \setminus E$, then $\Bold x$ is a separating point. So let $\Bold x \in \Lf \setminus E$.

By Lemma \ref{lem-2endpoints} we can assume that $f$ admits more than two fixed points. Let $F$ be the set of 
fixed points admitted by $f$ and let $$\Bold F = \{ (p, p, \ldots ) : p \in F \}.$$ If $\Bold x \in \Bold F$, 
then by Lemma \ref{lem-no-extreme-pt}, Corollary \ref{cor-fixed-points}, Proposition \ref{lem-endpoints}, and 
Lemmas \ref{S-fixed} and \ref{N-fixed}, $\Bold x$ is a separating point of $\Lf$.

Suppose $\Bold x \not\in \Bold F$. By Lemma \ref{lem-claim1} (iv) and Lemma \ref{lem-claim2} (iv), there exists 
$n \in \N$ such that $\min(F) < x_n < \max(F)$. Then there are fixed points $c, c'$ such that $c < x_n < c'$ and 
$(c, c') \cap F = \E$. We consider three cases.

\begin{enumerate}
\item[(a)] $c = 0$ or $c' = 1$.  
\end{enumerate}
Suppose $c' = 1$. If $c$ is an S-type, M-type or W-type fixed point, or an accumulation point of $F$, then 
$$\Lf = \L([0, c], f \uhr [0, c]) \cup \L([c, 1], f \uhr [c, 1])$$ and $\Bold x \in \L([c, 1], f \uhr [c, 1])$. 
Since $\L([c, 1], f \uhr [c, 1])$ has exactly two fixed points and $\Bold x$ is not one of them, 
$\L([c, 1], f \uhr [c, 1])$ is an arc and $\Bold x$ is a separating point of $\L([c, 1], f \uhr [c, 1])$ and 
hence of $\Lf$.

If $c$ is an N-type fixed point, witnessed by $(e, c')$, then $e$ is a fixed point and $e$ is not an N-type fixed 
point, so $$\Lf = \L([0, e], f \uhr [0, e]) \cup \L([e, 1], f \uhr [e, 1])$$
$\Bold x \in \L([e, 1], f \uhr [e, 1])$, $\L([e, 1], f \uhr [e, 1])$ is an arc and $\Bold x$ is not an endpoint 
of $\L([e, 1], f \uhr [e, 1])$. So again, $\Bold x$ is a separating point of $\Lf$.

Similarly if $c=0$.

\begin{enumerate}
\item[(b)] $c = \min(F) \ne 0$, or $c' = \max(F) \ne 1$.
\end{enumerate}
Suppose $c' = \max(F) \ne 1$. Then $c'$ is an N-type fixed point and $c$ is either an S-type or an M-type fixed point, 
or an accumulation point of $F$. In any case
$$\Lf = \L([0, c], f \uhr [0, c]) \cup \L([c, 1], f \uhr [c, 1]),$$
$$\L([0, c], f \uhr [0, c]) \cap \L([c, 1], f \uhr [c, 1]) = \{ (c, c, \ldots ) \},$$
$\Bold x \in \L([c, 1], f \uhr [c, 1])$, and by Lemma \ref{lem-2endpoints}, $\L([c, 1], f \uhr [c, 1])$ is an arc. 
Thus $\Bold x$ is a separating point of $\Lf$. 

Similarly if $c = \min(F) \ne 0$.

\begin{enumerate}
\item[(c)] $c \ne 0$, $c' \ne 1$, $c \ne \min(F)$, $c' \ne \max(F)$.
\end{enumerate}
Either $c$ or $c'$ is not an N-type fixed point. Suppose $c'$ is not. Then 
$$\Lf = \L([0, c'], f \uhr [0, c']) \cup \L([c', 1], f \uhr [c', 1]),$$
$\Bold x \in \L([0, c'], f \uhr [0, c'])$, $c'$ is the maximum fixed point admitted by $f \uhr [0, c']$, and so 
the result follows as in case (b) above.

Similarly if $c$ is not an N-type fixed point
\end{proof}

We now show that if $f$ does admit a splitting sequence then $\Lf$ is not an arc.

\begin{lemma}\label{lem-splitting}
Let $f : [0, 1] \to [0, 1]$ be a surjective continuous function. If $f$ admits a splitting sequence  
then there is a nondegenerate continuum $C \subset \Lf$ and a sequence of nondegenerate continua 
$$\langle C_n \subset \Lf : n \in \N \rangle ,$$ $C_n \ne C$, such that $C_n \to C$ in the Hausdorff metric.
\end{lemma}

\begin{proof}
Suppose $\si = \ls T_n = [l_n, r_n] : n \in \N \rs$ is a splitting sequence and let $N$ be an infinite subset 
of $\N$ such that for each $n \in N$ there is a nondegenerate interval $S_n \subset [0, 1]$, 
$S_n \cap T_n \subset\{l_n,r_n\}$ and $f(S_n) = f(T_n)$.

For each $n \in N$ let $S_n^n = T_n$. If $j \ge n$ and $S_n^j \subset [0, 1]$ has been defined, choose an 
interval $S_n^{j+1} \subset [0, 1]$ such that $f(S_n^{j+1}) = S_n^j$. Since $f$ is surjective, $S^{j+1}_n$ 
exists as $\Gamma(f) \cap ([0,1] \times S_n^j)$ must have a component $C$ such that $\pi_j(C) = S_n^j$, and 
so we can let $S^{j+1}_n = \pi_{j+1}(C)$.

For each $j < n$ let $S_n^j = T_j$. It follows that $S_n^j$ is nondegenerate for each $n \in N$, $j \le n$.

For each $n \in N$ let $S^n = \L(S_n^m, f \upharpoonright S_n^m)$. Then $\{ L(\si) \} \cup \{ S^n : n \in \N \}$ 
is a collection of nondegenerate continua in $\Lf$. If $\Bold t \in L(\si)$ then for each $n \in \N$ there is 
a point $\Bold s^n \in S^n$ such that $s_j^n = t_j$ for every $j \le n$ and hence any neighbourhood of $\Bold t$ 
meets infinitely many sets $S^n$. Furthermore, any sequence $\{ \Bold s^n \in S^n : n \in \N \}$ has a limit point 
in $L(\si)$. It follows that $S^n \to L(\si)$ in the Hausdorff metric.
\end{proof}

\begin{theorem}\label{thm-main}
Suppose $f : [0, 1] \to [0, 1]$ is a continuous surjective function. Then $\Lf$ is an arc if and only if 
$\Lf$ does not admit a splitting sequence.
\end{theorem}

\begin{proof} 
By Lemmas \ref{lem-arc} and \ref{lem-splitting}.
\end{proof}

\end{document}